\documentclass[10pt]{amsart}

\usepackage[latin1]{inputenc}
\usepackage{amsmath}
\usepackage{amsfonts}
\usepackage{amssymb}
\usepackage{ytableau}
\usepackage[mathscr]{eucal}
\usepackage{diagrams}
\usepackage{xy,color,graphicx}
\usepackage{hyperref}
\xyoption{all}
\usepackage{array}
\usepackage{enumerate}
\usepackage{url}

\newcommand{\N}{\mathbb{N}}
\newcommand{\Q}{\mathbb{Q}}
\newcommand{\C}{\mathbb{C}}
\newcommand{\Z}{\mathbb{Z}}

\newtheorem{definition}{Definition}

\newtheorem{theorem}{Theorem}
\newtheorem{example}{Example}

\newtheorem{lemma}{Lemma}

\title{The Monstrous Moonshine Picture}
\author{Lieven Le Bruyn} 
\address{Department of Mathematics, University of Antwerp \\ 
 Middelheimlaan 1, B-2020 Antwerp (Belgium) \\ {\tt lieven.lebruyn@uantwerpen.be}}

\begin{document}
\sloppy

\maketitle

\begin{abstract}
We describe the finite subgraph $\mathfrak{M}$ of Conway's Big Picture \cite{Conway} required to describe all $171$ genus zero groups appearing in monstrous moonshine. We determine the local structure of $\mathfrak{M}$ and give a purely group-theoretic description of this picture, based on powers of the  conjugacy classes $24J$ and $8C$. We expect similar results to hold for umbral moonshine groups and give the details for the largest Mathieu group $M_{24}$.
\end{abstract}

\section{Introduction}

Conway's {\em big picture}, see \cite{Conway} or section~2, is the Hasse diagram of a wonderful poset structure on the cosets $PSL_2(\Z) \backslash PGL_2^+(\Q)$, classifying projective classes of commensurable integral lattices of rank two. It was introduced to understand discrete groups like the congruence subgroups $\Gamma_0(N)$ and their normalizers in $PSL_2(\mathbb{R})$, in particular the genus zero groups appearing in {\em monstrous moonshine} \cite{ConwayNorton}. Later, it turned out to be important in Connes' approach to the Riemann hypothesis, see \cite{BostConnes} or \cite{Connes}, as its vertices form a basis for the regular representation of the Bost-Connes algebra, see also \cite{Plazas} or \cite{LeBruyn}. This algebra encodes the action of the absolute Galois group on the roots of unity.

After choice of an ordered basis of $\Q^2$ one can identify $PSL_2(\Z) \backslash PGL_2^+(\Q)$ (non-canonically) with $\Q_+ \times \Q/\Z$ and one can define a {\em hyper-distance} function on the cosets taking values in $\N_+$. The edges in the big picture are then drawn  between any two cosets at prime hyper-distance from each other, see section~2.  It then turns out that the big picture $\mathfrak{P}$ is the rooted product $\ast_p T_p$ over all prime numbers $p$ and where $T_p$ is the free $p+1$-valent tree. One purpose of this paper is to clarify what we mean by this product as there is an amount of non-commutativity involved, reminiscent of the meta-commutation relations in the Hurwitz quaternions, see \cite{ConwaySmith} and lemma~\ref{meta}. Another is to make the connection between certain cosets, the so-called {\em number-like} classes $M \frac{g}{h}$ corresponding to couples $(M,\frac{g}{h}) \in \Q_+ \times \Q/\Z$ with $M \in \N_+$, and sets of roots of unity. Here, we will view the class $M \frac{g}{h}$ as a primitive $h$-th root of unity centered at the number-class $h.M$.

To illustrate these two ideas we will describe the {\em monstrous moonshine picture} $\mathfrak{M}$ which is the finite full subgraph of the big picture on the number-like classes needed to describe all $171$ moonshine groups of \cite{ConwayNorton}. As a warm-up, take John McKay's $E_8$-observation on the monster conjugacy classes:
\[
\xymatrix@=.4cm{
& & & & & 3C \ar@{-}[d] & & \\
1A \ar@{-}[r] & 2A \ar@{-}[r] & 3A \ar@{-}[r] & 4A \ar@{-}[r] & 5A \ar@{-}[r] & 6A \ar@{-}[r] & 4B \ar@{-}[r] & 2B}
\]
In \cite{Duncan} John Duncan constructed an extended $E_8$-diagram on the nine moonshine groups corresponding to these conjugacy classes
\[
\xymatrix@=.4cm{
& & & & & 3|3 \ar@{-}[d] & & \\
1 \ar@{-}[r] & 2+ \ar@{-}[r] & 3+ \ar@{-}[r] & 4+ \ar@{-}[r] & 5+ \ar@{-}[r] & 6+ \ar@{-}[r] & 4|2+ \ar@{-}[r] & 2}
\]
To understand these nine moonshine groups we just need a tiny fraction of the big picture. More precisely , the subgraph on these twelve number-like classes, where black (resp. red and green) edges correspond to the primes $2$ (resp. $3$ and $5$).
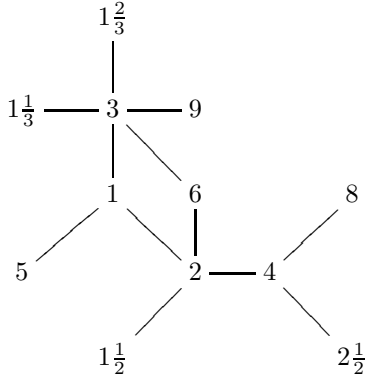
\[
\xymatrix@=.6cm{
& 1 \frac{2}{3} \ar@[red]@{-}[d] & & & \\
1 \frac{1}{3} \ar@[red]@{-}[r] & 3 \ar@[red]@{-}[r] \ar@[red]@{-}[d] \ar@{-}[rd] & 9 & & \\
& 1 \ar@[green]@{-}[ld] \ar@{-}[rd] & 6 \ar@[red]@{-}[d] & & 8 \ar@{-}[ld] \\
5 & & 2 \ar@{-}[r] \ar@{-}[ld] & 4 \ar@{-}[rd] & \\
& 1 \frac{1}{2} & & & 2 \frac{1}{2}}
\]
Each number-like class $M \frac{g}{h}$ is the projective class of the lattice $\Z(M \vec{e}_1 + \frac{g}{h} \vec{e}_2) \oplus \Z \vec{e}_2$. The stabilizer subgroup of a number-class $M$ is then the conjugate of $\Gamma = PSL_2(\Z)$ by $\alpha_M=${\tiny $\begin{bmatrix} M & 0 \\ 0 & 1 \end{bmatrix}$}. Conway's insight was to view $\Gamma_0(N)$ as the joint-stabilizer of the number-classes $1$ and $N$ as it is equal to $\Gamma \cap \alpha_N \Gamma \alpha_N^{-1}$. Therefore, $\Gamma_0(N)$ also stabilizes the {\em $(N|1)$-thread}  point-wise, that is, all number classes $M$ such that  $1 | M | N$. More generally, the group stabilizing the $(n|h)$-thread  (with $h | n$) pointwise is the conjugate $\alpha_h \Gamma_0(\tfrac{n}{h}) \alpha_h^{-1}$. The group $N+$ is obtained by adding to $\Gamma_0(N)$ the Atkin-Lehner involutions, which correspond to the symmetries of the$(N|1)$-thread, see subsection~\ref{snakes}. For example, the $(6|1)$-thread has four symmetries
\[
\xymatrix@=.6cm{3 \ar@{-}[r] \ar@[red]@{-}[d] & 6 \ar@[red]@{-}[d] \\
1 \ar@{-}[r] & 2} \]
generated by the horizontal and vertical flips, which correspond to multiplying the classes of lattices with the matrices  {\tiny $\frac{1}{\sqrt{2}} \begin{bmatrix} 2 & -1 \\ 6 & -2 \end{bmatrix}$} and {\tiny $\tfrac{1}{\sqrt{3}} \begin{bmatrix} 3 & 1 \\ 6 & 3 \end{bmatrix}$}, so with $6+$ we denote the subgroup of $PSL_2(\mathbb{R})$ generated by $\Gamma_0(6)$ and these two matrices. That is, the threads and their symmetries describe the groups $1,2,2+,3+,4+,5+$ and $6+$.

To describe groups like $n|h$ and $n|h+$ for $h$ dividing $n$ we consider the {\em $(n|h)$-serpent} which consists of all classes at hyper-distance a divisor of $h$ from the $(n|h)$-thread. These classes are fixed pointwise by $\Gamma_0(N)$ where $N=h.n$. The matrices $x=${\tiny $\begin{bmatrix} 1 & \frac{1}{h} \\ 0 & 1 \end{bmatrix}$} and $y=$ {\tiny $\begin{bmatrix} 1 & 0 \\ n & 1 \end{bmatrix}$} normalize $\Gamma_0(N)$ and we denote with $\Gamma_0(n|h)$ the group generated by $\Gamma_0(N)$ and $x$ and $y$. The moonshine group $n|h$ is then a specific subgroup of index $h$ in $\Gamma_0(n|h)$, that is, we have the situation
\[
\xymatrix{\Gamma_0(N) \ar[r] \ar@{.}@/_4ex/[rr]_{G} & n|h \ar[r] \ar@{.}@/^4ex/[r]^{h} & \Gamma_0(n|h)} \]
To describe the finite group $G$ and its index $h$ subgroup it is best to view the action of $x$ and $y$ on the $(n|h)$-serpent as power maps on different sets of roots of unity. For example, the $(3|3)$-serpent consists of the four classes at hyper-distance $3$ from $3$ which are best viewed as the vertices of a tetrahedron with center the class $3$.
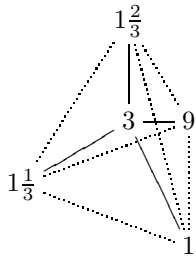
\[
\xymatrix@=.3cm{& & 1 \frac{2}{3} \ar@{.}[llddd] \ar@{.}[rdd] \ar@{.}[rdddd] \ar@[red]@{-}[dd] & \\
& & & \\
& & 3 \ar@[red]@{-}[rdd] & 9 \ar@{.}[llld] \ar@{.}[dd] \ar@[red]@{-}[l] \\
1 \frac{1}{3} \ar@[red]@{-}[rru] \ar@{.}[rrrd] & &  & \\
& & & 1} \]
The sets $\{ 1,1\frac{1}{3},1\frac{2}{3} \}$ and $\{ 9,1\frac{1}{3},1\frac{2}{3} \}$ give us two copies of $3$-rd roots of unity, centered at $3$, see subsection~\ref{roots}. The element $x$ acts as a power map on the first set fixing $9$ (so corresponds to a rotation with pole $9$ of the tetrahedron) and $y$ fixes $1$ and acts as a rotation on the second set. Thus, the finite group $G$ in this case is the rotation group of the tetrahedron, $A_4$, which has a unique subgroup of index $3$, giving us the moonshine group $3|3$, see example~\ref{tetra}. A similar story holds for the $(4|2)$-serpent
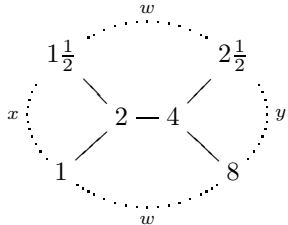
\[
\xymatrix@=.3cm{ 1 \frac{1}{2} \ar@{-}[rd] \ar@{.}@/_3ex/[dd]_x \ar@{.}@/^3ex/[rrr]^w & & & 2 \frac{1}{2} \ar@{-}[ld] \ar@{.}@/^3ex/[dd]^y \\
& 2 \ar@{-}[r] \ar@{-}[ld] & 4 \ar@{-}[rd] & \\
1  \ar@{.}@/_3ex/[rrr]_w & & & 8} \]
Here, $x$ interchanges the two square roots of one centered at $2$ and fixes those centered at $4$ and $y$ fixes the roots at $2$ and interchanges the roots at $4$, whence $G = C_2 \times C_2$ and in this case the index $2$ subgroup is generated by $x.y$ which gives us a description of $4|2$. To describe $4|2+$ we have to add the Atkin-Lehner involution $w$ determined by the symmetry of the $(4|2)$-thread with action as indicated on the other classes in the $(4|2)$-serpent.

We see that the subgraph of the big picture on the twelve number-classes given above contains all information needed to describe the nine moonshine groups of McKay's $E_8$-observation. The main purpose of this paper is to describe the {\em monstrous moonshine picture} $\mathfrak{M}$ which is the minimal subgraph of the big picture needed to describe {\em all} 171 moonshine groups.

It will turn out, see section~3, that $\mathfrak{M}$ is the subgraph on exactly $207$ number-like classes, among which there are $97$ number classes. For each of the number classes $C$ we determine the sets of primitive roots $M \frac{g}{h}$ centered at $C$, that is, $C=h.M$. This will allow us to describe the local structure of $\mathfrak{M}$ at $C$. For example, if $C$ is the center of $2$-nd, $3$-rd and $6$-th roots of unity, then $C=6M$ and the number-like classes in $\mathfrak{M}$ at hyper-distance dividing $6$ from $C$ are depicted below
\[
\xymatrix@=.3cm{& M \frac{1}{6} \ar@{-}[rd] & M \frac{2}{3} \ar@{-}[d] & 4M \frac{2}{3} \ar@{-}[ld] & \\
4M \ar@{-}[rd] & 3M \frac{1}{2} \ar@{-}[rd] & 2M \frac{1}{3} \ar@[red]@{-}[d] & 12 M \ar@{-}[ld] & M \frac{1}{3} \ar@{-}[ld] \\
M \frac{1}{2} \ar@{-}[r] & 2M \ar@[red]@{-}[r] \ar@{-}[ld] & \color{blue}{6M} \ar@{-}[ld] \ar@[red]@{-}[d] \ar@[red]@{-}[r] & 2M \frac{2}{3} \ar@{-}[r] \ar@{-}[rd] & M \frac{5}{6} \\
M & 3M & 18 M \ar@{-}[ld] \ar@{-}[d] \ar@{-}[rd] & & 4M \frac{1}{3} \\
& 9M & 9M \frac{1}{2} & 36M & }
\]
Here we use the meta-commutation property, see lemma~\ref{meta2}, to start the path to a class at hyper-distance $6$ from $6M$ with an edge of length $3$. The meta-commutation relations then gives a procedure to add he remaining edges in $\mathfrak{M}$. In this case we should add the edges

\[
\xymatrix@=.3cm{& M \frac{1}{6} \ar@{-}[rd] & M \frac{2}{3} \ar@{-}[d] & 4M \frac{2}{3} \ar@{-}[ld] & \\
4M \ar@{-}[rd] & 3M \frac{1}{2} \ar@[red]@{-}[u] \ar@[red]@{-}[dl] \ar@[red]@{-}[rrrd] \ar@[red]@{-}[dddr] \ar@{-}[rd] & 2M \frac{1}{3} \ar@[red]@{-}[d] & 12 M \ar@{-}[ld] \ar@[red]@{-}[u] \ar@[red]@{-}[rdd] \ar@[red]@{-}@/_4ex/[ddd] \ar@[red]@{-}@/_3ex/[lll] & M \frac{1}{3} \ar@{-}[ld] \\
M \frac{1}{2} \ar@{-}[r] & 2M \ar@[red]@{-}[r] \ar@{-}[ld] & \color{blue}{6M} \ar@{-}[ld] \ar@[red]@{-}[d] \ar@[red]@{-}[r] & 2M \frac{2}{3} \ar@{-}[r] \ar@{-}[rd] & M \frac{5}{6} \\
M & 3M \ar@[red]@{-}[l] \ar@[red]@{-}[d] \ar@[red]@{-}[ruuu] \ar@[red]@{-}[rrruu] & 18 M \ar@{-}[ld] \ar@{-}[d] \ar@{-}[rd] & & 4M \frac{1}{3} \\
& 9M & 9M \frac{1}{2} & 36M & }
\]
The monstrous moonshine picture $\mathfrak{M}$ is the union of the $(n|k)$-serpents if $n|k+e,f,\hdots$ is a moonshine group. Conversely, one might ask whether it is possible to construct $\mathfrak{M}$ using only group-theoretic information of $\mathbf{M}$.

We give such a procedure starting from the conjugacy classes $24J$, responsible for the largest serpent in $\mathfrak{M}$ (see subsection~\ref{24J}), and $8C$, corresponding to the $(8|4)$-serpent of example~\ref{8C}. The powers of these classes give the conjugacy classes
\[
\xymatrix@=.3cm{& 24J \ar@{-}[ld] \ar@{-}[rd]& \\
12J \ar@{-}[d] \ar@{-}[rrd] & & 8F \ar@{-}[d] \\
6F \ar@{-}[rrd] \ar@{-}[d] & & 4D \ar@{-}[d] \\
3C \ar@{-}[rd] & & 2B \ar@{-}[ld] \\
& 1A &} \qquad \xymatrix@=.3cm{ \\ 8C \ar@{-}[d] \\ 4B \ar@{-}[d] \\ 2A \ar@{-}[d] \\ 1A} 
\]
The procedure to construct $\mathfrak{M}$ goes as follows: let $X$ be a conjugacy class of order $n$ and let $l$ be minimal such that the class $X^l$ is one of the conjugacy classes above of order $k$. Then, if $k$ is even we construct the $(n,\tfrac{k}{2})$-serpent, if $k=3$ the $(n|3)$-serpent and if $k=1$ the $(n|1)$-serpent. The union of the serpents thus obtained is equal to $\mathfrak{M}$.

The underlying process at work here are the observations on harmonics from \cite[\S6]{ConwayNorton} giving a  precise connection between the moonshine group associated to a conjugacy class and those associated to its powers. The conjugacy classes determined by $24J$ and $8C$ are exactly the harmonics of the three lowest order conjugacy classes $1A,2A$ and $2B$.

We expect similar results to hold for all umbral groups and give the details for the largest Mathieu group $M_{24}$.

\section{Conway's Big Picture} \label{bigpicture}

In \cite{Conway} John H. Conway introduced a picture that makes it easier to describe groups commensurable with the modular group $\Gamma = PSL_2(\Z)$, in particular the discrete groups appearing in monstrous moonshine \cite{ConwayNorton}. In this section we will recall the construction of this {\em big picture} and use it to define the $171$ {\em moonshine groups}, with a focus on examples and explicit calculations. 

\subsection{The moonshine groups}

Monstrous moonshine \cite{ConwayNorton} assigns to every conjugacy class of an element $g$ of the monster sporadic simple group $\mathbf{M}$ a specific genus zero subgroup $\Gamma_g$ of $PSL_2(\mathbb{R})$ which is {\em commensurable} with the modular group $\Gamma$ (that is, $\Gamma_g \cap \Gamma$ has finite index in $\Gamma_g$ and $\Gamma$). There are exactly $171$ such {\em moonshine groups}, usually given in Conway-Norton notation
\[
\Gamma_g = n|h+e,f,\hdots \]
see Figure~\ref{moonshinegroups} which is taken from \cite[Table 2]{ConwayNorton}. Here, we have that

{\tiny
\begin{figure} \label{moonshinegroups}
\[
\begin{array}{|c|l||c|l||c|l||c|l||c|l||c|l|}
\hline
1A & 1 & 10B & 10+5 & 18B & 18+ & 26A & 26+ & 39A & 39+ & 60C & 60+4,15,60 \\
2A & 2+ & 10C & 10+2 & 18C & 18+9 & 26B & 26+26 & 39B & 39|3+ & 60D & 60+12,15,20 \\
2B & 2- & 10D & 10+10 & 18D & 18- & 27A & 27+ & 39CD & 39+39 & 60E & 60|2+5,6,30 \\
3A & 3+ & 10E & 10- & 18E & 18+18 & 27B & 27+ & 40A & 40|4+ & 60F & 60|6+10 \\
3B & 3- & 11A & 11+ & 19A & 19+ & 28A & 28|2+ & 40B & 40|2+ & 62AB & 62+ \\
3C & 3|3 & 12A & 12+ & 20A & 20+ & 28B & 28+ & 40CD & 40|2+20 & 66A & 66+ \\
4A & 4+ & 12B & 12+4 & 20B & 20|2+ & 28C & 28+7 & 41A & 41+ & 66B & 66+6,11,66 \\
4B & 4|2+ & 12C & 12|2+ & 20C & 20+4 & 28D & 28|2+14 & 42A & 42+ & 68A & 68|2+ \\
4C & 4- & 12D & 12|3+ & 20D & 20|2+5 & 29A & 29+ & 42B & 42+6,14,21 & 69AB & 69+ \\
4D & 4|2- & 12E & 12+3 & 20E & 20|2+10 & 30A & 30+6,10,15 & 42C & 42|3+7 & 70A & 70+ \\
5A & 5+ & 12F & 12|2+6 & 20F & 20+20 & 30B & 30+ & 42D & 42+3,14,42 & 70B & 70+10,14,35 \\
5B & 5- & 12G & 12|2+2 & 21A & 21+ & 30C & 30+3,5.15 & 44AB & 44+ & 71AB & 71+ \\
6A & 6+ & 12H & 12+12 & 21B & 21+3 & 30D & 30+5,6,30 & 45A & 45+ & 78A & 78+ \\
6B & 6+6 & 12I & 12- & 21C & 21|3+ & 30E & 30|3+10 & 46AB & 46+23 & 78BC & 78+6,26,39 \\
6C & 6+3 & 12J & 12|6 & 21D & 21+21 & 30F & 30+2,15,30 & 46CD & 46+ & 84A & 84|2+ \\
6D & 6+2 & 13A & 13+ & 22A & 22+ & 30G & 30+15 & 47AB & 47+ & 84B & 84|2+6,14,21 \\
6E & 6- & 13B & 13- & 22B & 22+11 & 31AB & 31+ & 48A & 48|2+ & 84C & 84|3+ \\
6F & 6|3 & 14A & 14+ & 23AB & 23+ & 32A & 32+ & 50A & 50+ & 87AB & 87+ \\
7A & 7+ & 14B & 14+7 & 24A & 24|2+ & 32B & 32|2+ & 51A & 51+ & 88AB & 88|2+ \\
7B & 7- & 14C & 14+14 & 24B & 24+ & 33A & 33+11 & 52A & 52|2+ & 92AB & 92+ \\
8A & 8+ & 15A & 15+ & 24C & 24+8 & 33B & 33+ & 52B & 52|2+26 & 93AB & 93|3+ \\
8B & 8|2+  & 15B & 15+5 & 24D & 24|2+3 & 34A & 34+ & 54A & 54+ & 94AB & 94+ \\
8C & 8|4+ & 15C & 15+15 & 24E & 24|6+ & 35A & 35+ & 55A & 55+ & 95AB & 95+ \\
8D & 8|2-  & 15D & 15|3 & 24F & 24|4+6 & 35B & 35+35 & 56A & 56+ & 104AB & 104|4+ \\
8E & 8- & 16A & 16|2+ & 24G & 24|4+2 & 36A & 36+ & 56B & 56|4+14 & 105A & 105+ \\
8F & 8|4- & 16B & 16- & 24H & 24|2+12 & 36B & 36+4 & 57A & 57|3+ & 110A & 110+ \\
9A & 9+ & 16C & 16+ & 24I & 24+24 & 36C & 36|2+ & 59AB & 59+ & 119AB & 119+ \\
9B & 9- & 17A & 17+ & 24J & 24|12 & 36D & 36+36 & 60A & 60|2+ & & \\
10A & 10+ & 18A & 18+2 & 25A & 25+ & 38A & 38+ & 60B & 60+ & & \\
\hline
\end{array}
\]
\caption{The Moonshine Groups}
\end{figure}
}
\begin{itemize}
\item{$n$ is the order of  $g \in \mathbf{M}$,}
\item{$h$ is a divisor of $24$ and of $n$,}
\item{$x=e,f,\hdots$ are exact divisors of $m=\tfrac{n}{h}$, that is $(x,\frac{m}{x})=1$.}
\end{itemize}

If $h=1$ we drop $|h$ from the notation and we abbreviate with $+$ the set of all exact divisors of $\tfrac{n}{h}$ and we write a minus sign if there are no exact divisors involved in the description of $\Gamma_g$.

The congruence subgroup $\Gamma_0(N)$ consists of the images in $\Gamma$ of all matrices of the form 
{\tiny $\begin{bmatrix} a & b \\ cN & d \end{bmatrix}$} with $a,b,c,d \in \Z$ and determinant $ad-Nbc=1$.

 From now on we will assume that $N=n.h$ with $h$ a divisor of $24$ and of $n$. Then the set of all matrices in $SL_2(\Q)$ of the form
{\tiny $\begin{bmatrix} a & \frac{b}{h} \\ cn & d \end{bmatrix}$}, with $a,b,c,d \in \Z$ and determinant $ad-\frac{n}{h}bc=1$, 
form a group and its image in $PSL_2(\Q)$ will be denoted by $\Gamma_0(n|h)$. This group is the conjugate of $\Gamma_0(\frac{n}{h})$ by {\tiny $\begin{bmatrix} h & 0 \\ 0 & 1 \end{bmatrix}$}. Further note that $\Gamma_0(N)$ is a normal subgroup of $\Gamma_0(n|h)$.

If $e$ is an exact divisor of $m=\tfrac{n}{h}$, then the set $w_e$ of all matrices of the form {\tiny $\begin{bmatrix} ae & \frac{b}{h} \\ cn & de \end{bmatrix}$}, with $a,b,c,d \in \Z$ and determinant $e^2 ad - \frac{n}{h}bc = e$, form a single coset of $\Gamma_0(n|h)$. As these cosets satisfy the relations
\[
w_e^2 \equiv 1~\text{mod}~\Gamma_0(n|h),~\quad w_e.w_f \equiv w_f.w_e \equiv w_g~\text{mod}~\Gamma_0(n|h) \]
with $g=\tfrac{e}{(e,f)}.\tfrac{f}{(e,f)}$, they generate a subgroup of involutions in the normalizer of $\Gamma_0(n|h)$ in $PSL_2(\mathbb{R})$. They are called {\em Atkin-Lehner involutions} of $\Gamma_0(n|h)$ and traditionally one calls $w_{n/h}$ the {\em Fricke involution}.

For exact divisors $e,f,\hdots$ of $m=\tfrac{n}{h}$ one denotes $\Gamma_0(n|h)+e,f,\hdots$ for the group obtained from $\Gamma_0(n|h)$ by adjoining its particular Atkin-Lehner involutions $w_e,w_f,\hdots$. Further note that $\Gamma_0(N)$ is a normal subgroup of $\Gamma_0(n|h)+e,f,\hdots$.

The moonshine group $n|h+e,f,\hdots$ is then a specific subgroup of index $h$ in $\Gamma_0(n|h)+e,f,\hdots$, the kernel of the group-morphism
\[
\lambda : \Gamma_0(n|h)+e,f,\hdots \rTo \C^* \]
which is trivial on the normal subgroup $\Gamma_0(N)$ as well as on the Atkin-Lehner involutions $w_e,w_f,\hdots$ and such that 
\[
\lambda(\begin{bmatrix} 1 & \frac{1}{h} \\ 0 & 1 \end{bmatrix})=e^{-\frac{2 \pi i }{h}}, \quad \text{and} \quad \lambda(\begin{bmatrix} 1 & 0 \\ n & 1 \end{bmatrix}) = e^{\pm \frac{2 \pi i}{h}} \]
the last value with a $+$ sign if {\tiny $\begin{bmatrix} 0 & -1 \\ N & 0 \end{bmatrix}$} $\in \Gamma_0(n|h)+e,f,\hdots$, and with a $-$ sign otherwise. In \cite[\S 3]{ConwayMcKay} it is shown that $\lambda$ is indeed well-defined.

\subsection{The big picture}

The discrete groups introduced before are best understood via their action on projective classes of $2$-dimensional integral lattices.

Let $L_1 = \langle \vec{e}_1,\vec{e}_2 \rangle = \Z \vec{e}_1 \oplus \Z \vec{e}_2$ be a fixed $2$-dimensional integral lattice. A lattice $L$ is said to be {\em commensurable} with $L_1$ if their intersection has finite index in both of them. Two such lattices $L$ and $L'$ are in the same {\em projective class} if there is a non-zero rational number $\lambda$ such that $L' = \lambda.L$.

Any lattice $L$ commensurable with $L_1$ is in the same projective class as a unique lattice in {\em standard form}
\[
L_{M,\frac{g}{h}} = \langle M \vec{e}_1+ \frac{g}{h} \vec{e}_2,\vec{e}_2 \rangle = \Z (M \vec{e}_1+ \frac{g}{h} \vec{e}_2) \oplus \Z \vec{e}_2 \]
with $M$ a strictly positive rational number and $\frac{g}{h}$ is a proper fraction in its least terms, that is $0 \leq g < h$ with $(g,h)=1$. That is, the set of projective classes of $2$-dimensional lattices commensurable with $L_1$ can be identified with $\Gamma \backslash PGL_2^+(\Q) = \Q_+ \times \Q/\Z$. Here, $PGL_2^+(\Q)$ is the group of all elements in $PGL_2(\Q)$ with strictly positive determinant, and the bijection is given by the assignment
\[
\Q_+ \times \Q/\Z \rTo \Gamma \backslash PGL_2^+(\Q) \qquad \begin{bmatrix} M & \frac{g}{h} \\ 0 & 1 \end{bmatrix} \mapsto \Gamma.\begin{bmatrix} M & \frac{g}{h} \\ 0 & 1 \end{bmatrix} = L_{M,\frac{g}{h}} \]
see \cite[Prop. 2.6]{Duncan}.

If $M \in \N_+$ we omit the comma and write $L_{M \frac{g}{h}}$ and call the lattice {\em number-like}. If in addition $g=0$ we write $L_M$ and call the lattice a {\em number-lattice}. With $X=M,\frac{g}{h}$ we denote the projective class of $L_{M,\frac{g}{h}}$. It will be convenient to associate to this class the matrix
\[
\alpha_X = \begin{bmatrix} M & \frac{g}{h} \\ 0 & 1 \end{bmatrix} \in GL_2^+(\Q) \]
For two classes $X,Y \in \Q_+ \times \Q/\Z$ let $a_{XY}$ be the smallest positive rational number such that $a_{XY}.\alpha_X \alpha_Y^{-1} \in GL_2(\Z)$. The {\em hyper-distance} between the two classes is then
\[
\delta(X,Y) = det(a_{XY}.\alpha_X \alpha_Y^{-1}) \in \N_+ \]
Conway showed in \cite[p. 329]{Conway} (see also \cite[\S 2.5]{Duncan} and \cite[\S 2.2]{Plazas} for a slightly different treatment) that this definition is symmetric and that the $log$ of the hyper-distance is a proper distance function on the projective classes of all lattices commensurable with $L_1$.

Further, if $g \in PGL_2^+(\Q)$ such that $\alpha_X.g \in \Gamma.\alpha_{X'}$ and $\alpha_Y.g \in \Gamma.\alpha_{Y'}$ (here, and elsewhere, we write $\alpha_Z$ for its image in $PGL_2(\Q)$), then it follows immediately from the definition of hyper-distance that
\[
\delta(X,Y) = \delta(X.g,Y.g) = \delta(X',Y') \]
that is, the maps induced by by left multiplication by elements of $PSL_2(\Q)^+$ are isometries, where defined. 

\begin{definition} {\em Conway's big picture} $\mathfrak{P}$ is the graph having as vertices the projective classes of lattices commensurable with $L_1$, that is all elements of $\Q_+ \times \Q/\Z$, and there is an edge between classes $X$ and $Y$ if and only if
\[
\delta(X,Y) = p \]
with $p$ a prime number. Alternatively, $\mathfrak{P}$ is the Hasse diagram of the poset structure on $\Q_+ \times \Q/\Z$ with unique minimal element $1$ defined by
\[
X \leq Y \quad \text{iff} \quad \delta(1,X) \leq \delta(1,Y)~\text{and}~\delta(1,Y)=\delta(1,X).\delta(X,Y) \]
The subgraph $\mathfrak{C}$ on all classes of number-lattices is called {\em the big cell}. It is the Hasse diagram of the poset-structure on $\N_+$ given by division, that is, $M \leq N$ iff $M | N$.
\end{definition}

For a number-like class $X=M \frac{g}{h}$ we have that the hyper-distance from $1$ is equal to $h^2M$. In fact we have
\[
1 \leq hM \leq M\frac{g}{h} \]
and the number class $hM$ is the unique class in the big cell $\mathfrak{C}$ at minimal hyper-distance from $M \frac{g}{h}$.

Conway showed that each class $X=M,\frac{g}{h}$ has exactly $p+1$ neighbors at hyper-distance $p$ for a prime number $p$. These are the classes
\[
X_k = \frac{M}{p},\frac{g}{hp}+\frac{k}{p}~\text{mod}~1~\quad \text{for $0 \leq k < p$}~\quad~\text{and}~\quad X_p=pM,\frac{pg}{h}~\text{mod}~1 \]
It is convenient to consider these $p$-neighbors as the classes corresponding to the matrices obtained by multiplying $\alpha_{X}$ {\em on the left} with the matrices
\[
P_k = \begin{bmatrix} \frac{1}{p} & \frac{k}{p} \\ 0 & 1 \end{bmatrix}~\quad~\text{for $0 \leq k < p$, and}~\quad~P_p = \begin{bmatrix} p & 0 \\ 0 & 1 \end{bmatrix} \]
The classes at hyper-distance a $p$-power from $1$ form a $p+1$-valent tree, and the big picture $\mathfrak{P}$ itself 'factorizes' as a product of these $p$-adic trees, see \cite[p. 332]{Conway}. 

We will now clarify the relevance of the big picture $\mathfrak{P}$ for the groups appearing in monstrous moonshine.

\begin{definition} Let $\mathcal{X}=\{ X_i~|~i \in I \}$ be a set of classes in the big picture $\mathfrak{P}$. We define the following subgroups of $PGL_2^+(\Q)$
\[
\Gamma_0(\mathcal{X}) = \{ g \in PGL_2^+(\Q)~|~\forall i \in I~:~~\alpha_{X_i}.g \in \Gamma.\alpha_{X_i} \} \]
and
\[
N(\mathcal{X}) = \{ g \in PGL_2^+(\Q)~|~\forall i \in I, \exists j \in I~:~\alpha_{X_i}.g \in \Gamma.\alpha_{X_j} \} \]
as, respectively, the {\em point-wise} and {\em set-wise} stabilizer subgroups of $\mathcal{X}$. 
\end{definition}

\begin{lemma} \label{lem1} With these notations, we have (for $n,h,N \in \N_+$)
\begin{enumerate}
\item{$\Gamma_0(\{ 1 \}) = \Gamma$ and $\Gamma_0(\{ X \}) = \alpha_X^{-1}.\Gamma.\alpha_X$.}
\item{$\Gamma_0(\{ 1,N \}) = \Gamma_0(N) = \Gamma_0(\{ e~:~e | N \})$.}
\item{$\Gamma_0(\{ n,h \}) = \Gamma_0(n|h)= \Gamma_0(\{ e~:~h | e | n \})$}
\end{enumerate}
\end{lemma}

\begin{proof} $(1)$ follows from the definition. $(2)$ follows as $\Gamma \cap \alpha_N.\Gamma.\alpha_N^{-1} = \Gamma_0(N)$. The second part then follows from the isometric property and the fact that there is a unique number lattice in $\{ e : 1 | e | N \}$ at hyper-distance $k$ from $1$ and $\tfrac{N}{k}$ from $N$. $(3)$ is proved in a similar manner.
\end{proof}

\subsection{Snakes and serpents} \label{snakes}

Because of lemma~\ref{lem1} we call for $h|n$ the subgraph of $\mathfrak{P}$ consisting of the number-lattices $\{ e~:~h|e|n \}$ the {\em $(n|h)$-thread}. More generally, if $X \geq Y$ then we call the {\em $(X|Y)$-thread} the subgraph on the lattices $\{ Z~|~X \geq Z \geq Y \}$. 

It follows from the factorization of the big picture $\mathfrak{P}$ that for $\delta(X,Y)=N$ and any divisor $k | N$ there is a unique lattice $Z$ in the $(X,Y)$-thread such that $\delta(X,Z)=k$ and $\delta(Y,Z)=\tfrac{N}{k}$. If we denote $\Gamma_0(X|Y)$ for $\Gamma_0(\{ X,Y \})$, then this group is also the point-wise stabilizer of the $(X|Y)$-thread.

The $(N|1)$-thread has a symmetry group of order $2^k$ where $k$ is the number of distinct prime divisors of $N$. These symmetries correspond to the Atkin-Lehner involutions $w_e$ introduced before for exact divisors $e$ of $N$. If $N=e.f$, then a number-lattice $k=x.y$ with $x | e$ and $y | f$ in the $(N|1)$-thread is send under this involution to the number-lattice $w_e(k)=\tfrac{e}{x}.y$. 

Indeed, for $w_e=${\tiny $\begin{bmatrix} ae & b \\ cN & de \end{bmatrix}$} with determinant $e$ we have that the image of $\alpha_k.w_e.\alpha_{w_e(k)}^{-1}$ in $PGL_2(\Q)$ is equal to
\[
\frac{1}{x}.\begin{bmatrix} xy & 0 \\ 0 & 1 \end{bmatrix}.\begin{bmatrix} ae & b \\ efc & de \end{bmatrix}.\begin{bmatrix} \frac{x}{ey} & 0 \\ 0 & 1 \end{bmatrix} = \begin{bmatrix} xa & yb \\ \frac{fc}{y} & \frac{de}{x} \end{bmatrix} \in \Gamma \]
A similar story applies to the $(n|h)$-thread. There are $2^l$ symmetries for the $(n|h)$-thread where $l$ is the number of distinct primes in $m=\tfrac{n}{h}$. These symmetries again correspond to the Atkin-Lehner involutions, which are in turn the conjugates by {\tiny $\begin{bmatrix} h & 0 \\ 0 & 1 \end{bmatrix}$} of the corresponding involutions of the $(\frac{n}{h}|1)$-thread as
\[
\begin{bmatrix} h^{-1} & 0 \\ 0 & 1 \end{bmatrix}.\begin{bmatrix} ae & b \\ c \frac{n}{h} & de \end{bmatrix}.\begin{bmatrix} h & 0 \\ 0 & 1 \end{bmatrix} = \begin{bmatrix} ae & \frac{b}{h} \\ cn & de \end{bmatrix} \]
Observe that these involutions also act as symmetries on the full $(N|1)$-thread where $N=h.n$. This is the picture for the groups $\Gamma_0(n|h)+e,f,\hdots$ for all $h | n$. 

We have seen that $\Gamma_0(N)$ fixes all classes in the $(N|1)$-thread, but it may fix more lattices. In fact, by \cite[Theorem p.336]{Conway} all classes stabilized by $\Gamma_0(N|1)$ are number-like and consists exactly of the classes $M \frac{g}{h}$ where $h$ is a divisor of $24$ such that $h^2 | N$ and $1 | M | \frac{N}{h^2}$. The subgraph of $\mathfrak{P}$ on this set of lattices Conway calls the {\em $(N|1)$-snake}. 

\begin{definition} Let $h$ be a divisor of $24$ and of $n$, and let $N=h.n$. We call the {\em $(n|h)$-serpent} the subgraph of the big picture $\mathfrak{P}$ on the number-like classes 
\[
M\frac{g}{k} \qquad \text{with} \quad k | h~\text{and}~1 | M | \frac{N}{k^2} \]
In particular, the $(n|h)$-serpent contains the full $(N|1)$-thread and is a part of the $(N|1)$-snake. Further, note that the $(N|1)$-serpent coincides with the $(N|1)$-thread which is all we need of the $(N|1)$-snake to construct the moonshine groups $N+e,f,\hdots$.
\end{definition}

In the next subsection we will see that the $(n|h)$-serpent determines the moonshine group $n|h+e,f,\hdots$. As all moonshine groups contain a congruence subgroup $\Gamma_0(N)$ as a normal subgroup, it is important to recall the description of the normalizer of $\Gamma_0(N)$ in $PSL_2(\mathbb{R})$.

If $h$ is the largest divisor of $24$ such that $h^2 | N$, then Conway calls the {\em spine} of the $(N|1)$-snake the subgraph  on all classes whose hyper-distance to the periphery is equal to $h$. For $n=\frac{N}{h}$, the spine of the $(N|1)$-snake is equal to the $(n|h)$-thread. The upshot of this terminology is that the normalizer of $\Gamma_0(N)$ fixes the $(N|1)$-snake setwise and must then also fix the spine setwise, so must be equal to $\Gamma_0(n|h)+$, which is the Atkin-Lehner theorem.

\begin{example} \label{8C} The $(8|4)$-serpent, relevant for the moonshine group $8|4+$ associated to conjugacy class $8C$. Here, $N=32$, so for $k=1$ we have the $(32|1)$-thread, indicated in red. For $k=2$ we have the lattices $M \frac{1}{2~}$ for $M | \tfrac{32}{4}=8$ and for $k=4$ the lattices $M \frac{1}{4}$ and $M \frac{3}{4}$ for $M=1,2$.
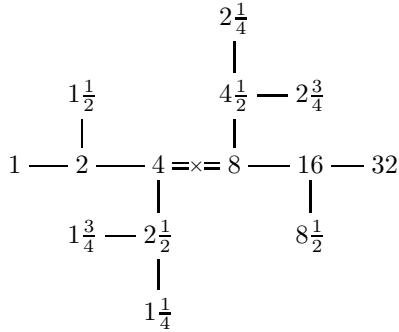
\[
\xymatrix@=.4cm{
& & & 2\frac{1}{4} \ar@{-}[d]  & & \\
& 1 \frac{1}{2} \ar@{-}[d] & & 4 \frac{1}{2} \ar@{-}[r] \ar@{-}[d] & 2 \frac{3}{4} & \\
1 \ar@[red]@{-}[r] & 2 \ar@[red]@{-}[r] & 4 \ar@{-}[d] \ar@[red]@{=}|\times[r] & 8 \ar@[red]@{-}[r] & 16 \ar@[red]@{-}[r] \ar@{-}[d] & 32 \\
& 1 \frac{3}{4} \ar@{-}[r] & 2 \frac{1}{2} \ar@{-}[d] & & 8 \frac{1}{2} & \\
& & 1 \frac{1}{4} & & & }
\]
The spine of the $(32|1)$-snake is the $(8|4)$-thread (indicated with a double red edge). The unique Atkin-Lehner involution involved is point-wise reflection over the middle of the $(8|4)$-thread (indicated with a $\times$). Note that in this case the $(8|4)$-serpent coincides with the $(32|1)$-snake.
\end{example}

\subsection{Roots of unity} \label{roots}

In noncommutative geometry, the big picture $\mathfrak{P}$ shows up as the basis of the regular representation of the Bost-Connes algebra $\mathbb{BC}$, see \cite[\S 4.3]{Plazas} and \cite[\S 2.4]{LeBruyn}. This algebra is at the heart of Connes' approach to the Riemann hypothesis and encodes the action of the abelianised Galois group $Gal(\overline{\Q}/\Q)^{ab}$ on the group of all roots of unity, which is compatible with the $\lambda$-ring structure on the representation rings of finite groups.

For this connection, it is useful to view a number-like class $M \frac{g}{h} \in \mathfrak{P}$ as a {\em primitive $h$-th root of unity centered at the number-lattice $h.M$}. Note that we have a full copy of the group $\pmb{\mu}_h$ of $h$-th roots of unity centered at $h.M$ consisting of the number-like classes $M \frac{e}{d}$ with $d | h$, which are all at hyper-distance $h$ from $h.M$.

Recall that the number of classes at hyper-distance $h$ from a fixed class is equal to $\psi(h)=h \prod_{p|h}(1+\tfrac{1}{p})$ where $\psi$ is Dedekind's psi function. Among the $\psi(h)$ classes at hyper-distance $h$ from $h.M$ there are several copies of $\pmb{\mu}_h$, depending on the chosen standard from of projective classes, all containing the same set of $\phi(h)$ classes corresponding to the canonical primitive $h$-th roots of unity centered at $h.M$, where $\phi(h)$ is the Euler function. What singles these classes out is the fact that on the path to them from $h.M$ we never used the operator $P_p$ for a prime $p$ dividing $h$.

For the construction of the moonshine groups $n|h+e,f,\hdots$ it is useful to consider also the {\em reverse canonical form} for a class $L_{M,\frac{g}{h}}$ which occurs by swapping the roles of the standard basis vectors $\vec{e}_1$ and $\vec{e}_2$. Then, the lattice $L_{M,\frac{g}{h}}$ lies in the same class as the lattice
\[
 \langle \frac{1}{h^2M} \vec{e}_2 + \frac{g'}{h} \vec{e}_1,\vec{e}_1 \rangle = \Z (\frac{1}{h^2M} \vec{e}_2 + \frac{g'}{h} \vec{e}_1) \oplus \Z \vec{e}_1 \]
with $g'$ the inverse of $g$ modulo $h$, and we write $(\frac{1}{h^2M},\frac{g'}{h})$ for the corresponding projective class. That is we will use these two notations for the same class in $\mathfrak{P}$
\[
X=M,\frac{g}{h} = (\frac{1}{h^2M},\frac{g'}{h}) \]
and use the matrix $\beta_X=$ {\tiny $\begin{bmatrix} 1 & 0 \\ \frac{g'}{h} & \frac{1}{M.h^2} \end{bmatrix}$} instead of the matrix $\alpha_X$.

As the character $\lambda$, used in defining $n|h+e,f,\hdots$ from $\Gamma_0(n|h)+e,f,\hdots$, is trivial on $\Gamma_0(N)$ and on the relevant Atkin-Lehner involutions it suffices in order to describe $n|h+e,f,\hdots$ to know the structure of the finite group
\[
\Gamma_0(n|h)^* = \Gamma_0(n|h)/\Gamma_0(N) \]
This group is generated by the action of $x=${\tiny $\begin{bmatrix} 1 & \frac{1}{h} \\ 0 & 1 \end{bmatrix}$} and $y=$ {\tiny $\begin{bmatrix} 1 & 0 \\ n & 1 \end{bmatrix}$} on the classes in the $(n|h)$-snake. The relevance of the introduction of the two sets of roots of unity centered at a number-class is that $x$ will act as a power-map on the first set, whereas $y$ acts as a power-map on the second set.

\begin{example} \label{tetra}  For $h=3$ the neighbors in the $3$-tree of $3M$ are the $4$-classes (with representations in the two standard forms)
\[
M \leftrightarrow (\frac{1}{M},0), \quad M \frac{1}{3} \leftrightarrow (\frac{1}{9M},\frac{1}{3}), \quad M \frac{2}{3} \leftrightarrow (\frac{1}{9M},\frac{2}{3}), \quad 9M \leftrightarrow (\frac{1}{9M},0) \]
It is convenient to view these classes as the vertices of a tetrahedron with center of gravity $3M\leftrightarrow (\frac{1}{3M},0)$.
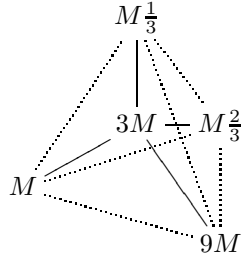
\[
\xymatrix@=.3cm{& & M \frac{1}{3} \ar@{.}[llddd] \ar@{.}[rdd] \ar@{.}[rdddd] \ar@[red]@{-}[dd] & \\
& & & \\
& & 3M \ar@[red]@{-}[rdd] & M \frac{2}{3} \ar@{.}[llld] \ar@{.}[dd] \ar@[red]@{-}[l] \\
M \ar@[red]@{-}[rru] \ar@{.}[rrrd] & &  & \\
& & & 9M} \]
The two sets of $3$-rd roots of unity, centered at $3M$ consist of the classes
\[
\{ M,M\frac{1}{3},M \frac{2}{3} \} \quad \text{and} \quad \{ 9M,M\frac{1}{3},M\frac{2}{3} \} \]
Power maps in the first copy of $\pmb{\mu}_3$ correspond to rotations with pole vertex $9M$ and power maps in the second copy are rotations with pole vertex $M$.

The $(3|3)$-serpent corresponds to the situation where $M=1$.In order to describe the moonshine group $3|3$ corresponding to conjugacy class 3C, we need to describe the action of $x = ${\tiny $\begin{bmatrix} 1 & \frac{1}{3} \\ 0 & 1 \end{bmatrix}$} on a class $M \frac{g}{h}$  is given by right-multiplication of $\alpha_{M\frac{g}{h}}$. As
\[
\alpha_{1}.x=\alpha_{1\frac{1}{3}},\quad \alpha_{1\frac{1}{3}}.x=\alpha_{1\frac{2}{3}},\quad \alpha_{1\frac{2}{3}}.x = \alpha_1,~\quad \text{and} \quad \alpha_{9}.x=\alpha_9 \]
That is, $x$ can be viewed as rotation with  a pole  through the class $9$. To study the action of $y =${\tiny  $\begin{bmatrix} 1 & 0 \\ 3 & 1 \end{bmatrix}$} it is best to use the second standard form and then the action of $y$  is given b right-multiplication of $\beta_{M\frac{g}{h}}$
\[
\beta_{9}.y = \beta_{1\frac{1}{3}}, \quad \beta_{1\frac{1}{3}}.y=\beta_{1\frac{2}{3}},\quad \beta_{1\frac{2}{3}}.y=\beta_9, \quad \text{and} \quad \beta_1.y=\beta_1 
\]
That is, $y$ is a rotation with pole the class $1$. Clearly, both rotations generate the full rotation symmetry group of the tetrahedron $A_4$. That is 
\[
\Gamma(3|3)^* = \Gamma_0(3|3)/\Gamma_0(9) \simeq A_4 \]
 and as this group has a unique subgroup of index $3$ generated by $x.y$ and $y.x$ it follows that the moonshine group $3|3$ is 
\[
3|3 = \langle \Gamma_0(9),\begin{bmatrix} 2 & \frac{1}{3} \\ 3 & 1 \end{bmatrix},\begin{bmatrix} 1 & \frac{1}{3} \\ 3 & 2 \end{bmatrix} \rangle \]
(see also \cite[example 2.9.1]{Duncan}).
\end{example}

A description of the groups $\Gamma_0(n|h)^{\ast}$ where $h$ is a divisor of $24$ and of $n$ was given in C. Ferenbaugh's thesis \cite{Ferenbaugh}, see also \cite[\S 3]{ConwayMcKay}.

\section{The moonshine picture} 

We have seen that, in order to understand the moonshine group $n|h+e,f,\hdots$, we need the $(n|h)$-serpent. For this reason we define {\em the monstrous moonshine picture} $\mathfrak{M}$ to be the subgraph of the big picture $\mathfrak{P}$ on all vertices of all $(n|h)$-serpents corresponding to the 171 moonshine groups $n|h+e,f,\hdots$.

We will also describe the {\em extended monstrous moonshine picture} $\mathfrak{M}^e$ to be the subgraph of the big picture $\mathfrak{P}$ on all vertices of all $(N|1)$-snakes corresponding to the 171 moonshine groups $n|h+e,f,\hdots$ with $N=h.n$.

\subsection{The anaconda} \label{24J}

The largest serpent, the {\em anaconda}, hiding in $\mathfrak{M}$ is the $(24|12)$-serpent determining the moonshine group $24|12$ which is associated to conjugacy class 24J of the monster $\mathbf{M}$. We will see that it consists of $70$ lattices, about one third of the total number of lattices in $\mathfrak{M}$. Note that the $(24|12)$-serpent is equal to the $(288|1)$-snake.

The anaconda's backbone is the $(288|1)$ thread below (edges in the $2$-tree are black, those in the $3$-tree red and the colored numbers are symmetric with respect to the $(24|12)$-spine. We give two number classes the same color if and only if they are the centers of similar sets of roots of unity, as we will show now.

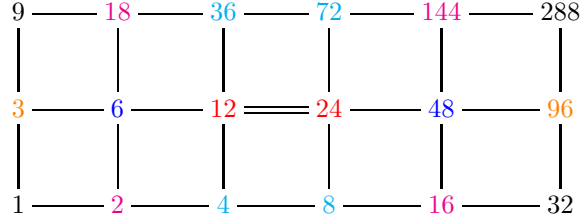
\[
\xymatrix{9 \ar@{-}[r] \ar@[red]@{-}[d] & \color{magenta}{18} \ar@{-}[r] \ar@[red]@{-}[d] & \color{cyan}{36} \ar@{-}[r] \ar@[red]@{-}[d] & \color{cyan}{72} \ar@{-}[r] \ar@[red]@{-}[d] & \color{magenta}{144} \ar@{-}[r] \ar@[red]@{-}[d] & 288 \ar@[red]@{-}[d] \\
\color{orange}{3} \ar@{-}[r] \ar@[red]@{-}[d] & \color{blue}{6} \ar@{-}[r] \ar@[red]@{-}[d] & \color{red}{12} \ar@{=}[r] \ar@[red]@{-}[d] & \color{red}{24} \ar@{-}[r] \ar@[red]@{-}[d] & \color{blue}{48} \ar@{-}[r] \ar@[red]@{-}[d] & \color{orange}{96} \ar@[red]@{-}[d] \\
1 \ar@{-}[r] & \color{magenta}{2} \ar@{-}[r] & \color{cyan}{4} \ar@{-}[r] & \color{cyan}{8} \ar@{-}[r] & \color{magenta}{16} \ar@{-}[r] & 32 } \]

Apart from the number classes, which are the divisors of $288$, we have to determine the number-like classes of the $(288|1)$-snake. These are the classes $M \frac{g}{h}$, with $h$ a divisor of $24$ such that $h^2$ divides $288$, and such that $(g,h) =1$. Because $288=2^5.3^2$ we have $h=1,2,4,6$ or $12$.

$h=1$ gives $M=$ $1, 2, 3, 4, 6, 8, 9, 12, 16, 18, 24, 32, 36, 48, 72, 96, 144, 288$.

$h=2$ gives the classes $M \frac{1}{2}$ for $M$ a divisor of $72$ :$1, 2, 3, 4, 6, 8, 9, 12, 18, 24, 36, 72$.

$h=3$ gives the classes $M \frac{1}{3}$ and $M \frac{2}{3}$ for $M$ a divisor of $32$ : $1, 2, 4, 8, 16, 32$.

$h=4$ gives the classes $M \frac{1}{4}$ and $M \frac{3}{4}$ for $M$ a divisor of $18$ : $1,2,3,6,9,18$.

$h=6$ gives the classes $M \frac{1}{6}$ and $M \frac{5}{6}$ for $M$ a divisor of $8$ : $1,2,4,8$.

$h=12$ gives the classes $M \frac{1}{12},M \frac{5}{12},M\frac{7}{12}$ and $M\frac{11}{12}$ for $M=1,2$.

This gives a total of $70$ classes. Next, we will focus on the center $C=M.h$ of each class $M \frac{g}{h}$ which we view as a primitive $h$-th root of unity centered at $M.h$.
{\Small
\[
\begin{array}{|c|l||c|l|}
\hline
C & h & C & h \\
\hline
1 & 1 & 18 & 1,2 \\
2 & 1,2 & 24 & 1,2,3,4,6,12 \\
3 & 1,3 & 32 & 1 \\
4 & 1,2,4 & 36 & 1,2,4 \\
6 & 1,2,3,6 & 48 & 1,2,3,6 \\
8 & 1,2,4 & 72 & 1,2,4 \\
9 & 1 & 96 & 1,3 \\
12 & 1,2,3,4,6,12 & 144 & 1,2 \\
16 & 1,2 & 288 & 1 \\
\hline
\end{array}
\]
}
Number classes $C$ having the same classes of primitive roots of unity centered there will have the same local structure in the anaconda. We will now describe the different types of local structures.

\subsection{The local structures} 

Our strategy to determine the structure of the moonshine picture(s) will be similar: for each $(n|k)$-serpent (or $(N|1)$-snake) we will first determine all number-like classes appearing in it. Next we will partition the number classes with respect to the set of primitive roots of unity centered at them. It then remains to determine, for a given set of $n$-roots of unity where $n | 24$, the local structure of the picture at the number-class.

To draw these local structures we will not draw all edges, but just one choice of path from the number class to the number-like class. We will now explain how to add the remaining edges. Whereas the big cell $\mathfrak{C}$ is commutative (being the Hasse diagram of the monoid $\N^{\times}_+$ partially ordered under division), the big picture $\mathfrak{P}$ is non-commutative, that is, if we add labels $P_i$ to the edges determined by the operator used to draw the edge (in the partial order direction) it is not always true that $P_i.Q_j = Q_jP_i$.

We have seen that the subgraph of $\mathfrak{P}$ consisting of all classes at hyper-distance a $p$-power from $1$ is a free $p+1$-valent tree $T_p$. Hence, the class $X$ at hyper-distance $p^k$ from $1$ corresponds to a unique product $P_{i_b}.\hdots.P_{i_2}.P_{i_1}P_p^a$ of the matrices $P_i$ with $0 \leq i  p$ and $P_p$ introduced before with $a+b=k$. Here, $p^a$ is the maximal number-class on the unique path from $X$ to $1$. 

That is, the matrices $P_i$ with $0 \leq i < p$ generated a free monoid and for all $0 \leq i < p$ we have $P_p.P_i = id$. In factoring $\mathfrak{P} = \ast_p~T_p$ we have to take into account that the matrices $P_i$ and $Q_j$ corresponding to different prime numbers $p$ and $q$ do not necessarily commute. Still, they satisfy a {\em meta-commutation relation} similar to that of factorization in the Hurwitz quaternions, see \cite[\S 5.2]{ConwaySmith}.

\begin{lemma} \label{meta2} For distinct primes $p$ and $q$ and for all $0 \leq i < p$ and $0 \leq j < q$ there exist unique $0 \leq k < p$ and $0 \leq l < q$ such that
\[
P_i.Q_j = Q_l.P_k \quad \text{and} \quad P_p.Q_j = Q_a.P_p~\text{with $a=pi~\text{mod}~q$} \]
\end{lemma}

\begin{proof}
\[
P_i.Q_j = \begin{bmatrix} \frac{1}{pq} & \frac{iq+j}{pq} \\ 0 & 1 \end{bmatrix} \]
and as $0 \leq iq+j < pq$ we have unique $0 \leq k < p$ and $0 \leq l < q$ such that $iq+j = lp+k$. \end{proof}

\begin{lemma} \label{meta} Any class $M,\frac{g}{h} \in \mathfrak{M}$ at hyper-distance $N=p^k.q^l.\dots r^s$ from $1$ can be identified uniquely with a product
\[
P_{i_1}P_{i_2} \hdots P_{i_b} Q_{j_1} Q_{j_2} \hdots Q_{j_d} \hdots R_{z_1} R_{z_2} \hdots R_{z_v} P_p^aQ_q^c \hdots R_r^u\]
for unique $0 \leq i_a < p$, $0 \leq j_b < q, \hdots, 0 \leq z_u < r$ and with $a+b=k,c+d=l,\hdots,u+v=s$.  Here $K = P_p^aQ_q^b\hdots R_r^u$ is the unique number-class in $\mathfrak{C}$ of minimal hyper-distance from $M,\frac{g}{h}$. For a number-like class $M \frac{g}{h}$ we have $K=Mh$ and $N=Mh^2$.

\end{lemma}

\begin{proof} A path from $1$ to $M,\frac{g}{h}$ of minimal length can be viewed as a product (left multiplication) $X_l X_{l-1} \hdots X_2 X_1$ with each $X_i$ one of the matrices $P_a,Q_b,\hdots,R_u$. The claim follows from using the meta-commutation relations of the previous lemma.
\end{proof}

For applications in moonshine, we only need the meta-commutation relations for the operators $2_0,2_1$ with respect to $3_0,3_1,3_2$ (note that these two sets each generate a free monoid). These are:
\[
2_0.3_0=3_0.2_0,~2_0.3_1=3_0.2_1,~2_0.3_2=3_1.2_0, \]
\[
~2_1.3_0=3_1.2_1,~2_1.3_1=3_2.2_0,~2_1.3_2=3_2.2_1 \]
These relations make it easy to add the forgotten edges in the pictures below. Further, for convenience, we start a path of hyper-length divisible by $3$ with the edges labeled $3_i$.

Here are the local structures appearing in the (extended) monstrous moonshine picture. For each type we have colored the central number class with a different color. Note further that if the set of roots is $\{ h_1,\hdots,h_k \}$, then the central number class must be a multiple of the least common multiple of the $h_i$.

\begin{itemize}
\item{Type $\{ 1,2 \}$: 
\[
\xymatrix@=.3cm{M  \ar@{-}[r] & \color{cyan}{2M} \ar@{-}[r] & M \frac{1}{2}} \]}
\item{Type $\{ 1,3 \}$:
\[
\xymatrix@=.3cm{& M \frac{1}{3} \ar@[red]@{-}[d] & \\
M \ar@[red]@{-}[r] & \color{orange}{3M} \ar@[red]@{-}[r] \ar@[red]@{-}[d] & M \frac{2}{3} \\
& 9M & } \]}
\item{Type $\{ 1,2,3 \}$:
\[
\xymatrix@=.3cm{3M \frac{1}{2} \ar@{-}[rd] & 2M \frac{1}{3} \ar@[red]@{-}[d] & 12M \ar@{-}[ld] \\
2M \ar@[red]@{-}[r] & \color{purple}{6M} \ar@[red]@{-}[r] \ar@[red]@{-}[d] \ar@{-}[ld] & 2M \frac{2}{3} \\
3M & 18M &} \]}
\item{Type $\{ 1,2,4 \}$:
\[
\xymatrix@=.3cm{ &  M \frac{1}{2} \ar@{-}[d] & & M \frac{1}{4} \ar@{-}[d] & \\
M \ar@{-}[r] & 2M \ar@{-}[r] & \color{magenta}{4M} \ar@{-}[r] \ar@{-}[d] & 2M \frac{1}{2} \ar@{-}[r] & M \frac{3}{4} \\
& & 8M \ar@{-}[ld] \ar@{-}[rd] & & \\
& 4M \frac{1}{2} & &  16 M &} \]}
\item{Type $\{ 1,2,3,6 \}$:
\[
\xymatrix@=.3cm{& M \frac{1}{6} \ar@{-}[rd] & M \frac{2}{3} \ar@{-}[d] & 4M \frac{2}{3} \ar@{-}[ld] & \\
4M \ar@{-}[rd] & 3M \frac{1}{2} \ar@{-}[rd] & 2M \frac{1}{3} \ar@[red]@{-}[d] & 12 M \ar@{-}[ld] & M \frac{1}{3} \ar@{-}[ld] \\
M \frac{1}{2} \ar@{-}[r] & 2M \ar@[red]@{-}[r] \ar@{-}[ld] & \color{blue}{6M} \ar@{-}[ld] \ar@[red]@{-}[d] \ar@[red]@{-}[r] & 2M \frac{2}{3} \ar@{-}[r] \ar@{-}[rd] & M \frac{5}{6} \\
M & 3M & 18 M \ar@{-}[ld] \ar@{-}[d] \ar@{-}[rd] & & 4M \frac{1}{3} \\
& 9M & 9M \frac{1}{2} & 36M & }
\]}
\item{Type $\{ 1,2,4,8 \}$:
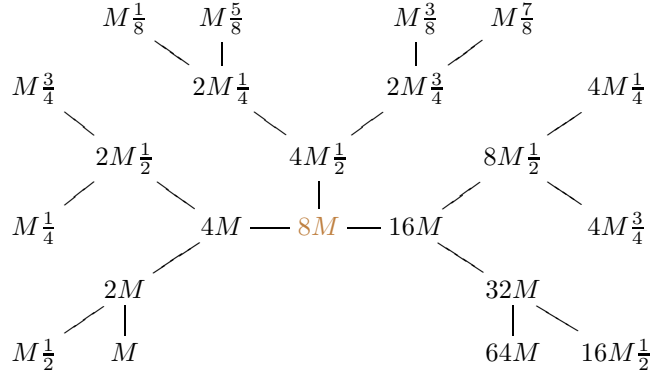
\[
\xymatrix@=.3cm{& M \frac{1}{8} \ar@{-}[rd] & M \frac{5}{8} \ar@{-}[d] & & M \frac{3}{8} \ar@{-}[d] & M \frac{7}{8} \ar@{-}[ld] & \\
M \frac{3}{4} \ar@{-}[rd] & & 2M \frac{1}{4} \ar@{-}[rd] & & 2M \frac{3}{4} \ar@{-}[ld] & & 4M \frac{1}{4} \ar@{-}[ld] \\
& 2M \frac{1}{2} \ar@{-}[rd] \ar@{-}[ld] & & 4M \frac{1}{2} \ar@{-}[d] & & 8M \frac{1}{2} \ar@{-}[ld] \ar@{-}[rd] \\
M \frac{1}{4} & & 4M \ar@{-}[r] \ar@{-}[ld] & \color{brown}{8M} \ar@{-}[r] & 16M \ar@{-}[rd] & & 4M \frac{3}{4} \\
& 2M \ar@{-}[ld] \ar@{-}[d] & & & & 32M \ar@{-}[d] \ar@{-}[rd] & \\
M \frac{1}{2} & M & & & & 64M & 16M \frac{1}{2} } 
\]}
\item{Type $\{ 1,2,3,4,6,12 \}$:
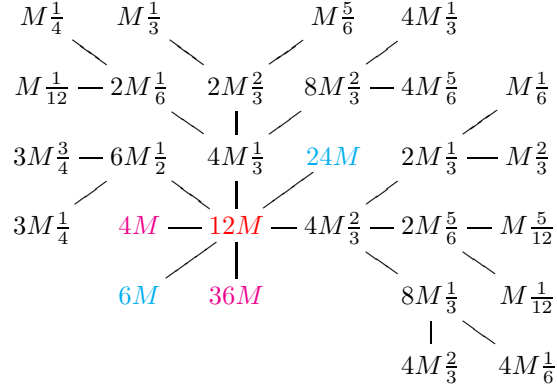
\[
\xymatrix@=.3cm{M \frac{1}{4} \ar@{-}[rd] & M \frac{1}{3} \ar@{-}[rd] & & M \frac{5}{6} \ar@{-}[ld] & 4M \frac{1}{3} \ar@{-}[ld] & \\
M \frac{1}{12} \ar@{-}[r] & 2M \frac{1}{6} \ar@{-}[rd] & 2M \frac{2}{3} \ar@{-}[d] & 8M \frac{2}{3} \ar@{-}[ld] \ar@{-}[r] & 4M \frac{5}{6} & M \frac{1}{6} \ar@{-}[ld] \\
3M \frac{3}{4} \ar@{-}[r] & 6M \frac{1}{2} \ar@{-}[ld] \ar@{-}[rd] & 4M \frac{1}{3} \ar@[red]@{-}[d] & \color{cyan}{24M} \ar@{-}[ld] & 2M \frac{1}{3} \ar@{-}[ld] \ar@{-}[r] & M \frac{2}{3} \\
3M \frac{1}{4} & \color{magenta}{4M} \ar@[red]@{-}[r] & \color{red}{12M} \ar@{-}[ld] \ar@[red]@{-}[r] \ar@[red]@{-}[d] & 4M \frac{2}{3} \ar@{-}[r] \ar@{-}[rd] & 2M \frac{5}{6} \ar@{-}[r] \ar@{-}[rd] & M \frac{5}{12} \\
& \color{cyan}{6M} & \color{magenta}{36M} & & 8M \frac{1}{3} \ar@{-}[d] \ar@{-}[rd] & M \frac{1}{12} \\
& & & & 4M \frac{2}{3} & 4M \frac{1}{6}} \]
Here we simplified the picture by indicating the local type of the intersection of the the number-classes $4M,6M,24M$ and $36M$ with the neighborhood of $12M$.}
\end{itemize}

\subsection{The (extended) monstrous moonshine picture} As mentioned before, in order to describe the (extended) monstrous moonshine picture $\mathfrak{M}$ (resp. $\mathfrak{M}^e$) we work through the list of 171 moonshine groups $n|h+e,f,\hdots$ as given in \cite[p. 327-329]{ConwayNorton} (that is, Figure~\ref{moonshinegroups} above) and determine for each of them all classes appearing in the $(n|h)$-serpent (resp. in the $(N|1)$-snake where $N=h.n$). The number classes of $\mathfrak{M}$ and $\mathfrak{M}^e$ are the same and are all divisors of the occurring  $N$. For each of these number classes we next determine the set of divisors of $24$ for which primitive roots of unity appear with center the given number class. These sets will then determine the local structure in that number class of $\mathfrak{M}$ (resp. $\mathfrak{M}^e$).

\begin{theorem} The (extended) monstrous moonshine picture $\mathfrak{M}$ (resp. $\mathfrak{M}^e$) is the subgraph of $\mathfrak{P}$ on exactly 207 (resp. 218) classes.

 Among these there are exactly 97 number classes $C$. The remaining number-like classes $M \frac{g}{h}$ are interpreted as primitive $h$-th roots of unity centered at the number-class $C=M.h$. 

Figure~\ref{RootsCentered} lists for each $C$ the sets of primitive roots centered at $C$. Here, a bracketed entry means these roots only appear in $\mathfrak{M}^e$ and not in $\mathfrak{M}$.
\begin{figure} \label{RootsCentered}
{\Small
\[
\begin{array}{|c|l||c|l||c|l||c|l|}
\hline
C & h & C & h & C & h & C & h \\
\hline
1 & 1 & 25 & 1 & 52 & 1,2,4 & 94 & 1  \\
2 & 1,2 & 26 & 1,2 & 54 & 1 & 95 & 1 \\
3 & 1,3 & 27 & 1 & 55 & 1 & 96 & 1,3 \\
4 & 1,2,4 & 28 & 1,2,4 & 56 & 1,2,4 & 104 & 1,2,4 \\
5 & 1 & 29 & 1 & 57 & 1,3 & 105 & 1 \\
6 & 1,2,3,6 & 30 & 1,2,3,6 & 59 & 1 & 110 & 1 \\
7 & 1 & 31 & 1 & 60 & 1,2,3,6 & 112 & 1,(2) \\
8 & 1,2,4,(8) & 32 & 1,2 & 62 & 1 & 117 & 1 \\
9 & 1 & 33 & 1 & 63 & 1 & 119 & 1 \\
10 & 1,2 & 34 & 1,2 & 64 & 1 & 120 & 1,3 \\
11 & 1 & 35 & 1 & 66 & 1 & 126 & 1,2 \\
12 & 1,2,3,4,6,12 & 36 & 1,2,4 & 68 & 1,2 & 136 & 1 \\
13 & 1 & 38 & 1 & 69 & 1 & 144 & 1,2 \\
14 & 1,2 & 39 & 1,3 & 70 & 1 & 160 & 1 \\
15 & 1,3 & 40 & 1,2 & 71 & 1 & 168 & 1 \\
16 & 1,2,(4) & 41 & 1 & 72 & 1,2,4 & 171 & 1 \\
17 & 1 & 42 & 1,2,3,(6) & 78 & 1 & 176 & 1 \\
18 & 1,2 & 44 & 1,2,(4) & 80 & 1,2 & 180 & 1,2 \\
19 & 1 & 45 & 1 & 84 & 1,2,3 & 208 & 1,2 \\
20 & 1,2,4 & 46 & 1 & 87 & 1 & 224 & 1 \\
21 & 1,3 & 47 & 1 & 88 & 1,2 & 252 & 1 \\
22 & 1,2 & 48 & 1,2,3,6 & 90 & 1,2 & 279 & 1 \\
23 & 1 & 50 & 1 & 92 & 1 & 288 & 1 \\
24 & 1,2,3,4,6,12 & 51 & 1 & 93 & 1,3 & 360 & 1 \\
& & & & & & 416 & 1 \\
\hline
\end{array}
\]
}
\caption{Roots centered at $C$}
\end{figure}
\end{theorem}

\vskip 3mm

\subsection{The harmonies of the moonshine picture}

We will now see that group-theoretic information of the monster group $\mathbf{M}$ determines the shape of $\mathfrak{M}$ and $\mathfrak{M}^e$, as well as the threads of all moonshine groups. The underlying reason is the theory of harmonics, see \cite[\S 6 and Table 3]{ConwayNorton}.

The 'anaconda' conjugacy class $24J$ appears to play a special role in moonshine, as does the conjugacy class $8C$. Powers of these elements belong to the conjugacy classes
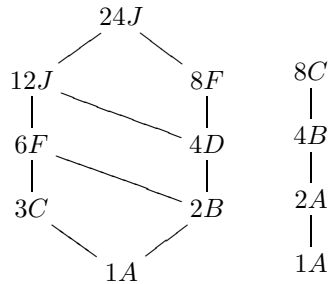
\[
\xymatrix@=.4cm{& 24J \ar@{-}[ld] \ar@{-}[rd]& \\
12J \ar@{-}[d] \ar@{-}[rrd] & & 8F \ar@{-}[d] \\
6F \ar@{-}[rrd] \ar@{-}[d] & & 4D \ar@{-}[d] \\
3C \ar@{-}[rd] & & 2B \ar@{-}[ld] \\
& 1A &} \qquad \xymatrix@=.4cm{ \\ 8C \ar@{-}[d] \\ 4B \ar@{-}[d] \\ 2A \ar@{-}[d] \\ 1A} 
\]

\begin{theorem} \label{poging1} Let $X$ be a conjugacy class of the monster $\mathbf{M}$ of order $n$ and let $Y$ be a conjugacy class of powers of $24J$ or $8C$ of maximal order $k$ such that $X$ belongs to the power-up classes of $Y$, see \cite{Atlas}. Then, for all but $12$ counter-examples the thread of the moonshine group corresponding to $X$ is 
\[
(n|\frac{k}{2})~\text{if $k$ is even,}~\quad (n|3)~ \text{if $Y=3C$, and $(n|1)$ if $Y=1A$}. \]
The counter-examples are exactly the power-up classes, see \cite{Atlas}, of
\[
\begin{array}{|c|l|}
\hline
8B & 24A, 24E,40B,88A,88B \\
8D & 24D,24H,40C \\
16A & 48A \\
16C & 32B \\
\hline
\end{array}
\]
Therefore, the thread of the moonshine group corresponding to conjugacy class $X$ of order $n$ is $(n|t)$ where $t$ is the value obtained from the above procedure, unless $X$ is in the power-up classes of $8B,8D,16A$ or $16C$ in which case we have the thread $(n|2t)$ except when $X=16C$.
\end{theorem}

\begin{proof} This follows from comparing the list of all 171 moonshine groups, given in \cite[p. 327-329]{ConwayNorton} (or Figure~\ref{moonshinegroups}) with the values obtained by this procedure, given in figure~\ref{threads}. Here, we write $!t$ for the correct value when it differs from the number obtained from the procedure.
\end{proof}

{\Small
\begin{figure} \label{threads}
\[
\begin{array}{|c|l||c|l||c|l||c|l||c|l||c|l|}
\hline
1A & 1 & 10B & 1 & 18B & 1 & 26A & 1 & 39A & 1 & 60C & 1 \\
2A & 1 & 10C & 1 & 18C & 1 & 26B & 1 & 39B & 3 & 60D & 1 \\
2B & 1 & 10D & 1 & 18D & 1 & 27A & 1 & 39DC & 1 & 60E & 1 \\
3A & 1 & 10E & 1 & 18E & 1 & 27B & 1 & 40A & 4 & 60F & 6 \\
3B & 1 & 11A & 1 & 19A & 1 & 28A & 2 & 40B & 1 (!2) & 62AB & 1 \\
3C & 3 & 12A & 1 & 20A & 1 & 28B & 1 & 40CD & 1 (!2) & 66A & 1 \\
4A & 1 & 12B & 1 & 20B & 2 & 28C & 1 & 41A & 1 & 66B & 1 \\
4B & 2 & 12C & 2 & 20C & 1 & 28D & 2 & 42A & 1 & 68A & 2 \\
4C & 1 & 12D & 3 & 20D & 2 & 29A & 1 & 42B & 1 & 69AB & 1 \\
4D & 2 & 12E & 1 & 20E & 2 & 30A & 1 & 42C & 3 & 70A & 1 \\
5A & 1 & 12F & 2 & 20F & 1 & 30B & 1 & 42D & 1 & 70B & 1 \\
5B & 1 & 12G & 2 & 21A & 1 & 30C & 1 & 44AB & 1 & 71AB & 1 \\
6A & 1 & 12H & 1 & 21B & 1 & 30D & 1 & 45A & 1 & 78A & 1 \\
6B & 1 & 12I & 1 & 21C & 3 & 30E & 3 & 46AB & 1 & 78BC & 1 \\
6C & 1 & 12J & 6 & 21D & 1 & 30F & 1 & 46CD & 1 & 84A & 2 \\
6D & 1 & 13A & 1 & 22A & 1 & 30G & 1 & 47AB & 1 & 84B & 2 \\
6E & 1 & 13B & 1 & 22B & 1 & 31AB & 1 & 48A & 1 (!2) & 84C & 3 \\
6F & 3 & 14A & 1 & 23AB & 1 & 32A & 1 & 50A & 1 & 87AB & 1 \\
7A & 1 & 14B & 1 & 24A & 1 (!2) & 32B & 1 (!2) & 51A & 1 & 88AB & 1 (!2) \\
7B & 1 & 14C & 1 & 24B & 1 & 33A & 1 & 52A & 2 & 92AB & 1 \\
8A & 1 & 15A & 1 & 24C & 1 & 33B & 1 & 52B & 2 & 93AB & 3 \\
8B & 1 (!2) & 15B & 1 & 24D & 1 (!2) & 34A & 1 & 54A & 1 & 94AB & 1 \\
8C & 4 & 15C & 1 & 24E & 3 (!6) & 35A & 1 & 55A & 1 & 95AB & 1 \\
8D & 1 (!2) & 15D & 3 & 24F & 4 & 35B & 1 & 56A & 1 & 104AB & 4 \\
8E & 1 & 16A & 1 (!2) & 24G & 4 & 36A & 1 & 56B & 4 & 105A & 1 \\
8F & 4 & 16B & 1 & 24H & 1 (!2) & 36B & 1 & 57A & 3 & 110A & 1 \\
9A & 1 & 16C & 1 & 24I & 1 & 36C & 2 & 59AB & 1 & 119AB & 1 \\
9B & 1 & 17A & 1 & 24J & 12 & 36D & 1 & 60A & 2 & & \\
10A & 1 & 18A & 1 & 25A & 1 & 38A & 1 & 60B & 1 & & \\
\hline
\end{array}
\]
\caption{Threads of Moonshine groups}
\end{figure}
}

We do not need the exact threads of all conjugacy classes in order to determine the (extended) monstrous moonshine picture. In fact, the values obtained from the procedure in the previous theorem suffice to get all classes in $\mathfrak{M}$ by taking the union of all $(n|k)$-serpents for the obtained $(n|k)$ from the procedure. 

For the extended moonshine picture $\mathfrak{M}^e$ we proceed similarly by taking the union of all $(N|1)$-snakes from the computed $(n|k)$ with $N=h.k$ from the procedure. This gives all classes except for the primitive $8$-th roots centered at $8$, which come from the $(64|1)$-snake determined by moonshine group $32|2+$ associated to conjugacy class $32B$ and the primitive $4$-th roots of unity, centered at $44$. These classes are contained in the $(176|1)$-snake, determined by the moonshine group $88|2+$ associated to conjugacy classes $88AB$.

\begin{theorem} The monstrous moonshine picture $\mathfrak{M}$ is the subgraph of the big picture $\mathfrak{P}$ on exactly the classes $M \frac{g}{h}$ contained in the $(n|k)$-serpents where $n$ is the order of an element in $\mathbf{M}$ and $k$ is a divisor of $24$ obtained as follows: among the conjugacy classes of order $n$ let $X$ be the class with a power-class $Y$ of maximal order $l$ among the powers of $24J$ or $8C$. Then, $k=3$ if $Y=3C$, $k=1$ if $Y=1A$ and $k=\tfrac{l}{2}$ if $l$ is even.
\end{theorem}

Conjugacy class $24J$ is responsible for the $(24|4)$-serpent of subsection~\ref{24J} and $8C$ for the $(8|4)$-serpent of example~\ref{8C}.

The relevance of these two conjugacy classes and the effectiveness of the procedure of theorem~\ref{poging1} is explained by the observations on harmonics \cite[\S 6]{ConwayNorton} which give the relationship between moonshine groups associated to a conjugacy class $g$ and those associated to the power map classes $g^d$.

If the conjugacy class $g$ corresponds to moonshine group $n|h+e,f,\hdots$ then the conjugacy class of $g^d$ has associated moonshine group
\[
n'|h'+e',f',\hdots \qquad \text{with} \quad n' = \frac{n}{(n,d)},~h'=\frac{h}{(h,d)} \]
and $e',f',\hdots$ are the divisors of $\tfrac{n'}{h'}$ among the numbers $e,f,\hdots$.

In the special case when the power $d$ divides $h$, one call $g$ the {\em $d$-th harmonic} of $g^d$, and elements $g$ with corresponding $h=1$ are called {\em fundamental} elements.

The conjugacy classes appearing as powers of the classes $24J$ and $8C$ are exactly the conjugacy classes which appear as harmonics of the first three conjugacy classes $1A,2A$ and $2B$, see \cite[Table 3]{ConwayNorton}. We have for these classes the harmonics
\[
\begin{array}{|c|l|}
\hline
1A & 1A,3C \\
2A & 2A,4B,8C \\
2B & 2B,4D,6F,8F,12J,24J \\
\hline
\end{array}
\]
and the procedure follows using the fact that the threads corresponding to $24J, 8C$ and $3C$ are resp. $(24|12)$, $(8|4)$ and $(3|3)$.

\section{The Mathieu Moonshine Picture}

One expects a similar story to hold for umbral moonshine, see \cite{ChengDuncan}. In this section we work out the details for the largest sporadic Mathieu group $M_{24}$, see \cite{Atlas2}.

The richest moonshine for $M_{24}$ is Mathieu moonshine, which associates a mock modular form of weight $\frac{1}{2}$ to each element $g$ of $M_{24}$. Given $g \in M_{24}$ we can ask for the invariance group of the corresponding mock modular form. There are characters involved, but allowing for these the largest group that preserves the form corresponding to $g$ is $\Gamma_0(n)$ where $n$ is the order of $g$.

In general, these groups are not of genus zero. Nonetheless, Mathieu moonshine does entail a natural association of genus zero groups to the conjugacy classes of $M_{24}$. This is explained in full detail in \cite{ChengDuncan2} for all cases of umbral moonshine, but only for the identity trace/graded dimension function in each case. But, it turns out that all the trace functions reappear amongst the various graded dimension functions (these are the 'multiplicative relations') so this paper can in principle be used to figure out which genus zero groups are attached to which conjugacy classes. I thank John Duncan for explaining this to me.

In the table below we give for each conjugacy class of $M_{24}$ its cycle length, the associated genus zero subgroup, and when possible the Niemeier root system of which the corresponding genus zero group is conjugated to that of the conjugacy class, see \cite[Table 1]{ChengDuncan2} and \cite[Table 3]{ChengDuncan}.

\[
\begin{array}{|c|l|l|l|| c | l | l | l |}
\hline
1A & 1^{24} & 2- & A_1^{24} & 7AB & 1^37^3 & 14+7 & D_8^3  \\
2A & 1^82^8 & 4- & A_3^8 & 8A & 1^2.2.4.8^2 & 16- & A_{15}D_9  \\
2B & 2^{12} & 4|2- & A_1^{24} & 10A & 2^210^2 & 20|2+5 & D_6^4 \\
3A & 1^63^6 & 6+3 & D_4^6 & 11A & 1^211^2 & 22+11 & D_{12}^2 \\
3B & 3^8 & 6|3 & A_1^{24} & 12A & 2.4.6.12 & 24|2 + 3 & . \\
4A & 2^44^4 & 8|2- & A_3^8 & 12B & 12^2 & 24|12- & A_1^{24} \\
4B & 1^42^24^4 & 8- & A_2^7D_5^2 & 14AB & 1.2.7.14 & 28+7 & .  \\
4C & 4^6 & 8|4- & A_1^{24} & 15AB & 1.3.5.15 & 30+3,5,15 & . \\
5A & 1^45^4 & 10+5 & D_6^4 & 21AB & 3.21 & 42|3+7 & D_8^3 \\
6A & 1^22^23^26^2 & 12+3 & .  & 23AB & 1.23 & 46+23 & D_{24} \\
6B & 6^4 & 12|6- & A_1^{24} & & & & \\
\hline
\end{array}
\]

As in the case of the monster group, we define the {\em Mathieu moonshine picture} $\mathfrak{M}_{24}$ to be the subgraph of the big picture $\mathfrak{P}$ on all classes belonging to the $(n|k)$-serpents, when $n|k+e,f,\hdots$ is a genus $0$ group appearing in Mathieu moonshine. The subgraph on all classes belonging to $(N|1)$-snakes for $N=h.n$ will then be the {\em extended Mathieu moonshine picture} $\mathfrak{M}_{24}^e$. Applying the same strategy as above, we obtain

\begin{theorem} The (extended) Mathieu moonshine picture $\mathfrak{M}_{24}$ (resp. $\mathfrak{M}_{24}^e$) is the subgraph of $\mathfrak{P}$ on exactly 93 (resp. 94) classes.

 Among these there are exactly 35 number classes $C$. The remaining number-like classes $M \frac{g}{h}$ are interpreted as primitive $h$-th roots of unity centered at the number-class $C=M.h$. 

The table below lists for each $C$ the sets of primitive roots centered at $C$. Here, a bracketed entry means these roots only appear in $\mathfrak{M}_{24}^e$ and not in $\mathfrak{M}_{24}$.
\[
\begin{array}{|c|l||c|l||c|l||c|l|}
\hline
C & h & C & h & C & h & C & h \\
\hline
1 & 1 & 10 & 1,2 & 22 & 1 & 46 & 1 \\
2 & 1,2 & 11 & 1 & 23 & 1 & 48 & 1,2,3,6 \\
3 & 1,3 & 12 & 1,2,3,4,6,12 & 24 & 1,2,3,4,6,12 & 63 & 1 \\
4 & 1,2,4 & 14 & 1,(2) & 28 & 1 & 72 & 1,2,4 \\
5 & 1 & 15 & 1 & 30 & 1 & 96 & 1,3 \\
6 & 1,2,3,6 & 16 & 1,2 & 32 & 1 & 126 & 1 \\
7 & 1 & 18 & 1,2 & 36 & 1,2,4 & 144 & 1,2 \\
8 & 1,2,4 & 20 & 1,2 & 40 & 1 & 288 & 1 \\
9 & 1 & 21 & 1,3 & 42 & 1,3 & & \\
\hline
\end{array}
\]
\end{theorem}

Here again, the largest serpent in the picture is the 'anaconda' $(24|12)$-serpent (or the $(288|1)$-snake) corresponding to conjugacy class $12B$. Also here it turns out that the powers of $12B$ allow us to compute the threads of most groups, and that this procedure gives a group-theoretic description of the pictures $\mathfrak{M}_{24}$ and $\mathfrak{M}_{24}^e$.

The powers of $12B$ are the following conjugacy classes, see \cite{Atlas2}.

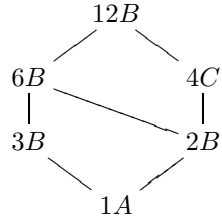
\[
\xymatrix@=.4cm{& 12B \ar@{-}[dl] \ar@{-}[dr] & \\
6B \ar@{-}[d] \ar@{-}[drr] & & 4C \ar@{-}[d] \\
3B \ar@{-}[dr] & & 2B \ar@{-}[dl] \\
& 1A &} \]

\begin{theorem} \label{poging2} Let $X$ be a conjugacy class of the Mathieu group $M_{24}$ of order $n$ and let $Y$ be a conjugacy class of powers of $12B$ of maximal order $k$ such that $X$ belongs to the power-up classes of $Y$, see \cite{Atlas2}. Then, for all but two counter-examples the thread of the genus zero group corresponding to $X$ is 
\[
(2n|k)  \]
In the table below we give the computed values of $k$ and indicate with $(!t)$ the correct value of the thread.
\[
\begin{array}{|c|l||c|l|}
\hline
1A & 1 & 7AB & 1 \\
2A & 1 & 8A & 1 \\
2B & 2 & 10A & 2 \\
3A & 1 & 11A & 1 \\
3B & 3 & 12A & 1(!2) \\
4A & 1(!2) & 12B & 12 \\
4B & 1 & 14AB & 1 \\
4C & 4 & 15AB & 1 \\
5A & 1 & 21AB & 3 \\
6A & 1 & 23AB & 1 \\
6B & 6 & & \\
\hline
\end{array}
\]
The two counter-examples are exactly the power-up classes of $4A$, see \cite{Atlas}.
\end{theorem}

An immediate consequence of this result is that the (extended) Mathieu moonshine picture $\mathfrak{M}_{24}$ (and $\mathfrak{M}_{24}^e$) can be described group-theoretically starting from the powers of conjugacy class $12B$ by the procedure outline above.

\end{document}